\documentclass{default}

\title{Integral monodromy groups of Kloosterman sheaves}
\def\titleHead{Integral monodromy groups of Kloosterman sheaves}
\author{Corentin Perret-Gentil}

\address{ETH Zürich, Department of Mathematics}
\curraddr{Centre de recherches mathématiques, Montréal, Canada}
\email{corentin.perretgentil@gmail.com}

\date{January 2018}


\begin{document}

\begin{abstract}
  We show that integral monodromy groups of Kloosterman $\ell$-adic sheaves of rank $n\ge 2$ on $\G_m/\F_q$ are as large as possible when the characteristic $\ell$ is large enough, depending only on the rank. This variant of Katz's results over $\C$ was known by works of Gabber, Larsen, Nori and Hall under restrictions such as $\ell$ large enough depending on $\car(\F_q)$ with an ineffective constant, which is unsuitable for applications. We use the theory of finite groups of Lie type to extend Katz's ideas, in particular the classification of maximal subgroups of Aschbacher and Kleidman-Liebeck. These results will apply to study hyper-Kloosterman sums and their reductions in forthcoming work.
\end{abstract}
\maketitle

\tableofcontents

\section{Introduction}

\subsection{Kloosterman sheaves}

For $z\in\C$, we write $e(z)=e^{2i\pi z}$. From now on, $q$ will denote a power of an \emph{odd} prime number, and $\F_q$ the field with $q$ elements

Kloosterman sums
\[\Kl_{2,p}(a)=\frac{-1}{\sqrt{p}}\sum_{x\in\F_p^\times} e \left(\frac{ax+x^{-1}}{p}\right) \ (a\in\F_p^\times)\]
were introduced by H.D. Kloosterman \cite{Kloo27} to study the number of representations of an integer by an integral positive-definite quarternary quadratic form. To achieve this, he proved the bound $|\Kl_{2,p}(a)|\le 2p^{1/4}$ for all $a\in\F_p^\times$, which can be refined to $|\Kl_{2,p}(a)|\le 2$ as a consequence of Weil's 1948 proof of the namesake conjectures in the case of curves over finite fields (see e.g. \cite[Theorem 11.11]{IK04}).

More generally, hyper-Kloosterman sums of rank $n\ge 2$
\begin{equation}
    \label{eq:Kln}
    \Kl_{n,q}(a)=\frac{(-1)^{n-1}}{q^{(n-1)/2}}\sum_{\substack{x_1,\dots,x_n\in\F_q^\times\\x_1\cdots x_n=a}} e \left(\frac{\tr(x_1+\dots+x_n)}{p}\right) \ (a\in\F_q^\times)
  \end{equation}
admit profound links with the spectral theory of automorphic forms through the Kuznetsov formula.

In \cite{Del1}, Deligne completed the proof of the Weil conjectures for algebraic varieties over finite fields, and gave in \cite{Del2} a generalization to weights of $\ell$-adic sheaves on such varieties\footnote{For the background and terminology on trace functions of $\ell$-adic sheaves on curves over finite fields, we refer the reader to \cite{KatzGKM} and \cite[Chapter 7]{KatzESDE}.}. Another deep result is the construction of an $\ell$-adic Fourier transform, corresponding to the discrete Fourier transform at the level of trace functions, with a control of the weights (see for example \cite[Chapters 5, 8]{KatzGKM}). In particular, this gives the following realization of hyper-Kloosterman sums as trace functions of $\ell$-adic sheaves:

\begin{theorem}[{Deligne; see \cite[Theorem 4.1.1]{KatzGKM} or \cite[Exposé 6, Théorème 7.8]{DelEC}}]\label{thm:Kln}
  Let $\ell\neq p$ be an auxiliary prime. For $n\ge 2$ an integer, there exists a middle-extension sheaf\,\footnote{For ease of notation, we will keep implicit the dependency of $\Klc_n$ on $q$ and $\ell$.} $\Klc_n$ of $\overline\Q_\ell$-modules on $\P^1/\F_q$ of rank $n$, corresponding to a continuous $\ell$-adic representation
  \[\rho_n: \pi_{1,q}:=\Gal\left(\F_q(T)^\sep/\F_q(T)\right)\to\GL\left(\Klc_{n,\overline\eta}\right)\cong\GL_n(\overline\Q_\ell),\]
  for $\Klc_{n,\overline\eta}$ the fiber at the generic point. Moreover, $\Klc_n$ is geometrically irreducible, lisse on $\G_m$, pointwise pure of weight $0$, $\Swan_\infty(\Klc_n)=1$, $\Swan_0(\Klc_n)=0$, and for any isomorphism of fields $\iota: \overline\Q_\ell\to\C$, the associated trace function
  \begin{eqnarray*}
    \F_q^\times&\to&\C\\
    a&\mapsto&\iota\circ\tr \left(\rho_n(\Frob_{a,q})\mid \Klc_{n,\overline\eta}\right)
  \end{eqnarray*}
  corresponds to the Kloosterman sum $\Kl_{n,q}$.
\end{theorem}

Notably, the purity claim implies by definition that $|\Kl_{n,q}(a)|\le n\text{ for all }a\in\F_q^\times$, which generalizes Weil's bound for Kloosterman sums of rank 2 to all $n\ge 2$.

\subsubsection{The Deligne-Katz equidistribution theorem}

This construction moreover allows to use powerful tools from algebraic geometry and representation theory to study Kloosterman sums and sums thereof. For example, Katz, building upon Deligne's equidistribution theorem, obtained the following result on the distribution of values of Kloosterman sums:
\begin{theorem}[Vertical Sato-Tate law, {\cite{KatzGKM}}]\label{thm:DeligneKatzST}
  Let $n\ge 2$ be an integer. Let $K$ be a maximal compact subgroup in $\SL_n(\C)$ if $n$ is odd and in $\Sp_n(\C)$ if $n$ is even, and let $\mu$ be the normalized Haar measure on $K$. When $q\to\infty$, the set $\{\Kl_{n,q}(a) : a\in\F_q^\times\}$ becomes equidistributed in $\tr(K)$ with respect to the measure $\tr_*\mu$.
\end{theorem}

\subsubsection{Monodromy groups}

A key ingredient for Theorem \ref{thm:DeligneKatzST} is the determination of the geometric and arithmetic \textit{monodromy groups}
\[G_\geom(\Klc_n)=\iota\overline{\rho_n\big(\pi_{1,q}^\geom\big)}\le G_\arith(\Klc_n)=\iota\overline{\rho_n\big(\pi_{1,q}\big)}\hspace{0.5cm}\le\GL_n(\C)\]
of the Kloosterman sheaf $\Klc_n$, where $\overline{\,\cdot\,}$ denotes Zariski closure and $\pi_{1,q}^\geom=\Gal(\F_q(T)^\sep/\overline\F_q(T))$.\\

Indeed, Deligne's theorem shows that under rather generic conditions (including the equality of $G_\geom$ and $G_\arith$), there is always an equidistribution result in a maximal compact subgroup of the monodromy group (see \cite[Chapters 3, 9]{KatzGKM} and \cite[Theorem 9.2.6, Theorem 9.6.10]{KatzSarnak91}). Katz's result is the following:

\begin{theorem}[{\cite[Chapter 11]{KatzGKM}}]\label{thm:monodromyKl}
  For $n\ge 2$, we have\footnote{Recall that we work in odd characteristic for simplicity. Katz also handles the case $p=2$.}
  \begin{equation}
  \label{eq:Emonodromy}
  G_\geom(\Klc_n)=G_\arith(\Klc_n)=
  \begin{cases}
    \SL_n(\C)&:n\text{ odd}\\
    \Sp_n(\C)&:n\text{ even}.
  \end{cases}
\end{equation}
\end{theorem}

Actually, the properties of the Kloosterman sheaves (see Proposition \ref{prop:propertiesKl} \ref{item:KldetTrivial}--\ref{item:Klpairingeven} below) show that \eqref{eq:Emonodromy} holds with equalities replaced by inclusions, so Theorem \ref{thm:monodromyKl} means that the monodromy groups are as large as possible.\\

A crucial ingredient in Katz's proof is the fact that $G_\geom(\Klc_n)^0$ is semisimple (which follows from a result of Deligne \cite[1.3]{Del2}, because the sheaf is pointwise pure of weight $0$).

\subsection{Integral Kloosterman sheaves}\label{subsec:integralKS}

\subsubsection{Kloosterman sums in cyclotomic integers}

Let us observe that the (unnormalized) Kloosterman sums $q^{\frac{n-1}{2}}\Kl_{n,q}$, a priori complex-valued, actually take values in the discrete subring $\Z[\zeta_p]$. This point of view was adopted by Fisher \cite{Fisher95} and Wan \cite{Wan95}.

To handle normalizations, we note that by the evaluation of quadratic Gauss sums
\[\sqrt{p}\in
  \begin{cases}
    \Z[\zeta_p]&: p\equiv 1\pmod{4}\\
    \Z[\zeta_{4p}]&: p\equiv 3\pmod{4},
  \end{cases}
\]
so that $\Kl_{n,q}$ takes values in $\Oc_{p^{(n-1)/2}}$, the localization at $p^{(n-1)/2}$ of the ring of integers $\Oc$ of
\begin{equation}
  \label{eq:fieldDef}  
  \begin{cases}
    \Q(\zeta_p)&n\text{ odd or }p\equiv 1\pmod{4}\\
    \Q(\zeta_{4p})&\text{otherwise}.
  \end{cases}
\end{equation}

Given a prime ideal $\lf\normal\Oc$ above a prime $\ell\neq p$, we can also study the reduction 
\[\Kl_{n,q} \pmod{\lf} : \F_q\to\F_\lf\]
in the residue field $\F_\lf=\Oc_{\sqrt{p}}/\lf\Oc_{\sqrt{p}}\cong\Oc/\lf$.

\subsubsection{Integral definition of the sheaves}

These observations happen to transfer at the level of sheaves. If $\lambda$ is the $\ell$-adic valuation on $\Oc$ corresponding to $\lf$, note that $\sqrt{p}\in\Oc_\lambda^\times$, so $\Kl_{n,q}$ takes values in $\Oc_\lambda$. Then we have:

\begin{theorem}[Deligne]\label{thm:KlnOlambda}
  Let $n\ge 2$ be an integer.
  \begin{itemize}
  \item The Kloosterman sheaf $\Klc_n$ from Theorem \ref{thm:Kln} can be defined as a sheaf of $\Oc_\lambda$-modules on $\P^1/\F_q$, corresponding to a continuous $\ell$-adic representation $\rho_n: \pi_{1,q}\to\GL_n(\Oc_\lambda)$.
  \item By reduction modulo $\lf$ in the residue field $\F_\lambda=\Oc_\lambda/\lf\Oc_\lambda$, we also obtain a sheaf $\widehat\Klc_n$ of $\F_\lambda$-modules corresponding to the representation $\hat\rho_n: \pi_{1,q}\to\GL_n(\F_\lambda)$, with trace function $\Kl_{n,q}\pmod{\lf}$.
  \end{itemize}
\begin{equation*}
  \xymatrix@C=1.8cm@R=0.5cm{
    &\GL_n(\Oc_\lambda)\ar[r]^{\hspace{0.4cm}\tr}\ar[dd]^\pi&\Oc_\lambda\ar[dd]^\pi&\\
    \pi_{1,q}\ar[ur]^{\rho_n}\ar[dr]_{\widehat\rho_n}&&&\ar[ul]_{\Kl_{n,q}}\ar[dl]^{\substack{\Kl_{n,q}\\\pmod{\lf}}}\F_q^\times\\
    &\GL_n(\F_\lambda)\ar[r]^{\hspace{0.4cm}\tr}&\F_\lambda&
  }
\end{equation*}
\end{theorem}
\begin{proof}
  This recursively follows from the fact that the $\ell$-adic Fourier transform can itself be defined at the level of $\Oc_\lambda$-modules, see \cite[Theorem 4.1.1, Chapter 5]{KatzGKM}.
\end{proof}

\subsubsection{Monodromy groups}
Our original motivation is to use the $\ell$-adic formalism to study the distribution of Kloosterman sums in cyclotomic integers and their reductions.

A first important step in this direction would be to determine the integral monodromy groups
\[G_\geom(\Klc_n)=\rho_n\big(\pi_{1,q}^\geom\big)\le G_\arith(\Klc_n)=\rho_n(\pi_{1,q})\le\GL_n(\Oc_\lambda)\]
and/or their reductions $G_\geom(\widehat \Klc_n)\le G_\arith(\widehat \Klc_n)\le \GL_n(\F_\lambda)$.

By Proposition \ref{prop:propertiesKl} \ref{item:KldetTrivial}--\ref{item:Klpairingeven} below, we still have
\[G_\geom(\Klc_n)\le G_\arith(\Klc_n)\le
\begin{cases}
  \SL_n(\Oc_{\lambda})&:n\text{ odd}\\
  \Sp_n(\Oc_{\lambda})&:n\text{ even}.
\end{cases}
\]
As Katz notes in the introduction of \cite{KatzGKM}, it is an interesting question to ask whether these integral monodromy groups are still equal and as large as possible, knowing that their Zariski closure in $\GL_n(\overline\Q_\ell)$ is $\SL_n(\overline\Q_\ell)$ (resp. $\Sp_n(\overline\Q_\ell)$) by Theorem \ref{thm:monodromyKl}.

The determination of integral and finite monodromy groups is usually more difficult than that of the monodromy groups over $\C$, since we consider simply subgroups of $\GL_n(\Oc_\lambda)$ or $\GL_n(\F_\lambda)$, instead of algebraic subgroups $G$ of $\GL_n(\C)$ with $G^0$ semisimple as before, and the structure of the former is much more complicated.\\

Our result is the following:

\begin{notation}
  We write $\ell\gg_n 1$ (resp. $\ell\ll_n 1$) for the condition that $\ell$ be larger (resp. smaller) than some constant depending only on $n$.
\end{notation}

\begin{theorem}\label{thm:KlcIntegralMono}
  Let $n\ge 2$ be coprime to $p$. For $\ell\gg_n 1$ with $\ell\equiv 1\pmod{4}$ and $\lambda$ an $\ell$-adic valuation on $\Oc=\Z[\zeta_{4p}]$ with $([\F_\lambda:\F_\ell], n)=1$, the Kloosterman sheaf $\Klc_n$ of $\Oc_\lambda$-modules over $\P^1/\F_q$ defined in Section \ref{subsec:integralKS} satisfies
  \begin{equation}
    \label{eq:integralMonodromy}
    G_\geom(\Klc_n)=G_\arith(\Klc_n)=
    \begin{cases}
      \SL_n(\Oc_{\lambda})&:n\text{ odd}\\
      \Sp_n(\Oc_{\lambda})&:n\text{ even}.
    \end{cases}
  \end{equation}
  In particular,
    \begin{equation}
    \label{eq:Flmonodromy}
  G_\geom(\widehat\Klc_n)=G_\arith(\widehat\Klc_n)=
  \begin{cases}
    \SL_n(\F_{\lambda})&:n\text{ odd}\\
    \Sp_n(\F_{\lambda})&:n\text{ even}.
  \end{cases}
\end{equation}
The same results hold true without restriction on $\ell\pmod{4}$ if $p\equiv 1\pmod{4}$ or $n$ is odd, with $\Oc=\Z[\zeta_p]$ \textup{(}see \eqref{eq:fieldDef}\textup{)}. 
\end{theorem}

\begin{remark}\label{rem:indexCylFieldRam}
  In particular, the results hold for all $\lambda$ of degree $1$ above a prime $\ell\gg_n 1$ (a set of natural density $1$). In general, note that $([\F_\lambda:\F_\ell], n)=\left(\ord(\ell\in\F_p^\times),n\right)\mid(p-1,n)$ (see \cite[Theorem 2.13]{Was97}).
\end{remark}

\begin{acknowledgements}
  The author would like to thank his supervisor Emmanuel Kowalski for guidance and advice during this project, Richard Pink for mentioning his results with Michael Larsen in \cite{LarsPink11}, and the referees for helpful comments. It is a pleasure to acknowledge the importance of the works we are building upon. This work was partially supported by DFG-SNF lead agency program grant 200021L\_153647, and the results also appear in the author's PhD thesis \cite{PG16}.
\end{acknowledgements}

\section{Existing results and heuristics}

\subsection{The results of Gabber and Nori}
In \cite[Chapter 12]{KatzGKM}, Katz presents the proof of the following result of Gabber:
\begin{theorem}[Gabber]
  If $\Z[\zeta_{4p}]_\lambda=\Z_\ell$, then \eqref{eq:integralMonodromy} holds if $\ell\gg_{n,p} 1$.
\end{theorem}

Unfortunately, the implicit constant depends on $p$ in an ineffective way. A similar result is shown by Nori \cite{Nori87}, with the same limitations.

For the applications in analytic number theory that we would like to consider, however, it would be necessary that \eqref{eq:integralMonodromy} holds for all $\ell$ large enough, independently from $p$.
\subsection{Consequences of the works of Larsen-Pink}
By results of Larsen and Pink (see \cite[Theorem 3.17]{Lars95} and the applications in \cite[Section 7]{Katz12}, \cite[pp. 155--156]{KowLargeSieve08}, \cite[p. 29]{KowLS06} and \cite[p. 7]{KowWeilNumbers06}), the monodromy result of Katz over $\C$ (Theorem \ref{thm:monodromyKl}) implies that for every $p$, there exists a set $\Lambda(n,p)$ of primes of natural density $1$ such that for all $\ell\in\Lambda(n,p)$, the result \eqref{eq:Flmonodromy} holds. Indeed, for fixed $n, q$ and varying $\lambda$, Kloosterman sheaves form a compatible system (see \cite[8.9]{KatzGKM}).

However, as for the results of Gabber and Nori, the way $\Lambda(n,p)$ is constructed is highly dependent on $p$. This is not a problem for the applications of Kowalski mentioned above, but issues arise if we need to take $\ell,p\to\infty$ with some restrictions on the range as in \cite{KowRank06}.

\subsection{(Invariant) generation of $\SL_n(\F_\ell)$}

In Section \ref{sec:localMonodromy}, we will see that $G_\geom(\Klc_{n})$ contains elements conjugate to
\begin{equation}
  \label{eq:mu}
  m=\left(
  \begin{matrix}
    0&0&\dots&0&(-1)^{n+1}\\
    1&0&\dots&0&0\\
    0&1&\dots&0&0\\
    \vdots&  &\ddots&\vdots&\vdots \\
    0&  &     &1&0
  \end{matrix}
\right)\text{ and }u=\left(
  \begin{matrix}
    1&1&&&\\
    &1&1&&\\
    &&1&\ddots&\\
    &&&\ddots&1\\
    &&&&1
  \end{matrix}
\right).
\end{equation}
As an indication that the monodromy group for $n$ odd should be $\SL_n(\F_\lambda)$, we have the following:
\begin{proposition}
  For $n$ odd, the elements $u$ and $m$ in \eqref{eq:mu} generate $\SL_n(\F_\ell)$.
\end{proposition}
\begin{proof}
  This can be obtained by proceeding in a way similar to \cite{GowTamb92}, considering the element $w=m^2u(mum)^{-1}$ and using induction on $n$.
\end{proof}

However, we do not know whether these elements are \textit{invariant generators}, namely whether any two conjugates are still generators. Without that, we may not conclude anything concerning Theorem \ref{thm:KlcIntegralMono}.

\subsection{The case $n=2$ and $\F_\lambda=\F_\ell$}

Hall \cite{Hall08} proved a classification theorem that generalizes a theorem of Yu on the $\F_\ell$-monodromy of hyperelliptic curves, and also applies to show big monodromy results for families of twists of elliptic curves, as needed in \cite{KowRank06}.

A particular case is the following, deduced from the classification of linear groups generated by transvections by Zalesski and Serezkin (a well-known result of Dickson when $n=2$):
\begin{theorem}[{\cite[Theorem 1.1]{Hall08}}]\label{thm:Hall1.1}
  Let $V$ be a $\F_\ell$-vector space with a perfect alternating pairing $V\times V\to\F_\ell$, and let $H\le\GL(V)$ be an irreducible primitive subgroup that preserves the pairing. If $H$ contains a transvection and $\ell\ge 3$, then $H=\Sp(V)$.
\end{theorem}

This can be applied to the sheaf $\Klc_2$ as follows:
\begin{proposition}\label{prop:HallMonoKl}
  Let $\ell\ge 3$ be a prime with $\ell\equiv 1\pmod{p}$, so that $\F_\lambda=\F_\ell$ (see Remark \ref{rem:indexCylFieldRam}). Then \eqref{eq:Flmonodromy} holds for $\widehat\Klc_2$: the arithmetic and geometric monodromy groups are equal to $\SL_2(\F_\ell)$.
\end{proposition}
\begin{proof}
  A unipotent element of drop $1$ is an element whose Jordan decomposition has exactly one Jordan block of size 2 and all other blocks trivial. Moreover, a transvection is an element of drop $1$ and determinant $1$. The result then follows from Proposition \ref{prop:propertiesKl} \ref{item:singleJordanBlockKl} and Theorem \ref{thm:Hall1.1}.
\end{proof}

However, this argument does not generalize to $n\ge 3$, since the image of the inertia at $0$ in $\Klc_n$ contains a transvection only when $n=2$ (see Proposition \ref{prop:propertiesKl}). Moreover, we cannot handle the case $\F_\lambda\neq\F_\ell$, since Hall considers only reduction of sheaves of $\Z_\ell$-modules (and not $\Z[\zeta_{4p}]_\lambda$-modules as for Kloosterman sheaves).

\section{Strategy and classification theorem over $\F_\lambda$}

\subsection{Equivalence of large $\Oc_\lambda$ and $\F_\lambda$-monodromy}
\begin{lemma}
  The properties \eqref{eq:integralMonodromy} and \eqref{eq:Flmonodromy} in Theorem \ref{thm:KlcIntegralMono} are equivalent.
\end{lemma}
\begin{proof}
  On one hand, \eqref{eq:integralMonodromy} implies \eqref{eq:Flmonodromy} by surjectivity of the reduction $\SL_n(R)\to\SL_n(R/\af)$ for any discrete valuation ring $R$ and any ideal $\af\normal R$ (since $\SL_n(R)$ is generated by elementary matrices in this case). On the other hand, an argument of Serre \cite[IV--23, 27--28]{Serre89} shows that the converse holds (see also \cite[8.13.3]{KatzESDE} for a result valid for general closed subgroups of $\GL_n(\Oc_\lambda)$).
\end{proof}
Understandably, we will prefer to work with finite groups of Lie type in \eqref{eq:Flmonodromy}, rather than with groups over complete rings in \eqref{eq:integralMonodromy}.

\subsection{Classification theorem over $\F_\lambda$}

The strategy of the proof of Theorem \ref{thm:KlcIntegralMono} is then as follows: if the conclusion does not hold, there exists a maximal (proper) subgroup
\[G_\geom(\widehat\Klc_n)\le H\lneq
\begin{cases}
  \SL_n(\F_\lambda)&:n\text{ odd}\\
  \Sp_n(\F_\lambda)&:n\text{ even}.
\end{cases}\]
By using the classification of maximal subgroups of classical groups of Aschbacher and Kleidman-Liebeck, we will show the following general classification theorem:
\begin{theorem}\label{thm:classification}
  Let $n\ge 2$, let $\F_\lambda$ be a field of characteristic $\ell$ and let
  \[H\le
  \begin{cases}
    \SL_n(\F_\lambda)&: n\text{ odd}\\
    \Sp_n(\F_\lambda)&: n\text{ even}
  \end{cases}
  \]
  be a maximal (proper) subgroup such that:
\begin{enumerate}
\item \label{item:irredfaithfulrepr} The action of $H$ on $\F_\lambda^n$ is irreducible.
\item \label{item:singleJordanBlock} $H$ contains a unipotent element with a single Jordan block.
\end{enumerate}
Then, for $\ell\gg_n 1$, we have either:
\begin{enumerate}
\item\label{item:classification1} $H=N_{\SL_n(\F_\lambda)}(\SO_n(\F_\lambda))$ for $n\ge 3$ odd\footnote{We recall that for $n$ odd, there is only one type of orthogonal group over a finite field up to isomorphism (see e.g. \cite[Section 2.5]{KleiLieb90}), so we do not need to specify the quadratic form.}.
\item\label{item:classification3} $H=N_{\SL_n(\F_\lambda)}(\SL_n(\F'))$ for $n\ge 3$ odd or $H=N_{\Sp_n(\F_\lambda)}(\Sp_n(\F'))$ for $n$ even, if $\F'\le\F_\lambda$ is a subfield of prime index.
\item\label{item:classification4} $H=N_{\SL_n(\F_\lambda)}(\SU_n(\F'))$ for $n\ge 3$ odd, if $\F'\le\F_\lambda$ is a subfield  of index $2$.
\end{enumerate}
\end{theorem}

More precisely, a theorem of Larsen-Pink on finite subgroups of algebraic groups and results on representations of finite groups of Lie type in various characteristics allow to reduce to the descent of a classification theorem of Suprunenko \cite[Theorem (1.9)]{Sup95} about subgroups of classical groups over algebraically closed fields containing regular unipotent elements (i.e. a unipotent element with a single Jordan block in the case of $\SL_n$).\\

By the properties of Kloosterman sheaves, $H$ verifies the hypotheses of Theorem \ref{thm:classification}, while at the same time conclusions \ref{item:classification1}--\ref{item:classification4} can be excluded. This would give Theorem \ref{thm:KlcIntegralMono}.

\subsection{Comparison with Katz's method over $\C$}

The above strategy is parallel to the approach Katz uses to prove his Theorem \ref{thm:monodromyKl}.

Indeed, by using the semisimplicity of $G_\geom^0(\Klc_n)\le\SL_n(\C)$ and the presence of a unipotent element with a single Jordan block, Katz reduces to the following classification theorem:
\begin{theorem}[{\cite[Classification Theorem 11.6]{KatzGKM}}]\label{thm:KatzClassification}
  Let $n\ge 2$ be an integer and $\Gc$ be a simple Lie algebra over $\C$ given with a faithful irreducible representation
  \[\rho: \Gc\hookrightarrow\sl_n(\C).\]
  Suppose that there exists a nilpotent element $N\in\Gc$ such that $\rho(N)$ has a single Jordan block. Then the pair $(\Gc,\rho)$ is isomorphic to one of the following:
  \begin{enumerate}
  \item $\Gc=\sl_n(\C)$ \textup{(}$n\ge 2$\textup{)}, $\sp_n(\C)$ \textup{(}$n\ge 4$ even\textup{)} or $\so_n(\C)$ \textup{(}$n\ge 5$ odd\textup{)}, with the standard $n$-dimensional representation.
  \item $\Gc=\mathfrak{g_2}(\C)$ with its unique 7-dimensional irreducible representation.
  \end{enumerate}
\end{theorem}

The extraneous cases are then excluded by using properties of Kloosterman sheaves.\\

Comparing the classification theorems \ref{thm:classification} and \ref{thm:KatzClassification}, it is interesting to note how the maximality assumption replaces the simplicity (it cannot be assumed a priori that $G_\geom(\widehat\Klc_n)\le\GL_n(\F_\lambda)$ is quasisimple), and how additional conclusions arise in Theorem \ref{thm:classification}, calling for new arithmetic inputs to exclude them.

\begin{remark}
  Theorem \ref{thm:KatzClassification} also follows from the classification theorem of Suprunenko mentioned above, which is valid over an algebraically closed field of arbitrary characteristic.
\end{remark}

\section{Local monodromy of Kloosterman sheaves}\label{sec:localMonodromy}

To obtain useful information for the determination of monodromy groups, a general strategy is to study images of inertia groups at singularities, which lie in the geometric monodromy group.

For example, the tame part of the break decomposition gives unipotent elements with prescribed Jordan form (see e.g. Proposition \ref{prop:propertiesKl} below), while values of Swan conductors can rule out the existence of certain morphisms (see e.g. \cite[Lemma 1.19]{KatzGKM}).\\

In this section, we recall results about local monodromy of Kloosterman sheaves from \cite{KatzGKM}, emphasizing that they still hold for the reduced sheaves, and explicitly compute the local monodromy at $\infty$. This will be useful in the next two sections.
\subsection{Results from \cite{KatzGKM}}

\begin{proposition}\label{prop:propertiesKl}
  Let $n\ge 2$ and for $A$ equal to $\overline\Q_\ell$, or $\F_\lambda$, let $\Klc_n$ be the Kloosterman sheaf of $A$-modules on $\P^1/\F_q$ defined in Theorems \ref{thm:Kln} and \ref{thm:KlnOlambda}, corresponding to a representation $\rho: \pi_{1,q}\to\GL_n(A)$. Then
  \begin{enumerate}
  \item \label{item:singleJordanBlockKl} $\Klc_n$ is unipotent as $I_0$-representation, where $I_0\le\pi_{1,q}$ is the inertia group at $0$. The image of a topological generator of the tame inertia group at $0$ has a single Jordan block.
  \item \label{item:KltotallyWildInfty} $\Klc_n$ is totally wild at $\infty$, with $\Swan_\infty(\Klc_n)=1$. In particular:
    \begin{enumerate}
    \item $\rho(I_\infty)$ acts irreducibly on $A^n$ and admits no faithful $A$-linear representation of dimension $<n$.
    \item\label{item:KltotallyWildInftyb} Any character $\rho(I_\infty)\to A^\times$ is trivial on $\rho(P_\infty)$, for $P_\infty\le\pi_{1,q}$ the wild inertia group at $\infty$.
    \end{enumerate}
  \item \label{item:KldetTrivial} $\det\Klc_n$ is trivial.
  \item \label{item:Klpairingeven} If $n$ is even, there exists an alternating perfect pairing $\Klc_n\otimes\Klc_n\to A$ of lisse sheaves on $\G_m$.
  \item \label{item:Klpairingodd} If $n$ is odd, then $\Klc_n\otimes\Klc_n$ is totally wild at $\infty$, with all breaks at $1/n$. In particular, there is no nonzero $P_\infty$-equivariant bilinear form $\Klc_n\otimes\Klc_n\to A$.
  \end{enumerate}
\end{proposition}

\begin{proof}
 
  \begin{enumerate}
  \item See \cite[7.4.1]{KatzGKM}. By \cite[12.3.3]{KatzGKM}, the result still holds for $\Klc_n$ as a sheaf of $\F_\lambda$-modules.
  \item This is \cite[1.11, 1.18]{KatzGKM} with the fact that $\Swan_\infty(\Klc_n)=1$.
  \item See \cite[7.4.3]{KatzGKM}.
  \item See \cite[4.1.11]{KatzGKM} (existence) and \cite[4.2.1]{KatzGKM} (sign).
  \item For the first assertion, see \cite[10.4.4]{KatzGKM}. For the second, proceed as in \cite[4.1.7]{KatzGKM}.
  \end{enumerate}
  In the finite case, see also \cite[12.3]{KatzGKM}.
\end{proof}

By Proposition \ref{prop:propertiesKl} \ref{item:singleJordanBlockKl}, note that the geometric monodromy group contains an element conjugate to the Jordan block $u$ from Equation \eqref{eq:mu}.

\subsection{Explicit local monodromy at $\infty$}

The local monodromy at $\infty$ of $\Klc_n$ is determined explicitly in \cite{KatzGKM} (as $P_\infty$-representation) and \cite{FKMSBilinearKS} (more precisely as $I_\infty$-representation), and more generally for hypergeometric sheaves in \cite[Proposition 0.7]{Fu10}. We make this even more concrete by finding a matrix form of the representation.

\begin{proposition}\label{prop:monodromyInftyExplicit}
  Under the same notations as in Proposition \ref{prop:propertiesKl}, we assume that $k=\F_q$ contains a primitive $2n$th root of unity $\zeta_{2n}$. Let $Z\in\overline{k(T)}$ be a solution to $Z^{2n}=T$ and $W\in\overline{k(Z)}=\overline{k(T)}$ be a solution to
  \[W^{|k|}-W=-Z^2.\]
  Then the restriction $\rho\mid_{I_\infty}: I_\infty\to\GL_n(A)$ is isomorphic (over $\F_\lambda$) to the representation
  \[\sigma\mapsto\left((-1)^{(n+1)\left(\frac{j-i+i_0}{n}\right)} e \left(\frac{n\tr_{k/\F_p}(a_0\zeta_{n}^{i})}{p}\right)\delta_{i-j\equiv i_0\pmod{n}}\right)_{1\le i,j\le n,}\]
  where $i_0\in\Z/2n$ and $a_0\in k$ are such that $\sigma(Z)=\zeta_{2n}^{i_0}Z$ and $\sigma(W)=\zeta_{2n}^{2i_0}(W+a_0)$.
\end{proposition}
\begin{remark}
  Assuming that $k$ contains a $n$th root of unity is not a constraint for our purpose. Indeed, if $L/k$ is a finite extension, we have a commutative diagram
\begin{equation*}
  \xymatrix@R=0.4cm{
    I_{\infty,L}=I_{\infty,k} \ar[r]&\pi_{1,L}^\geom=\pi_{1,k}^\geom\ar[dl]\ar[dr]&\\
    \pi_{1,L}\ar[dr]\ar[rr]&&\pi_{1,k}.\ar[dl]\\
    &\GL_n(\F_\lambda)&
  }
\end{equation*}
\end{remark}
\begin{proof}
  By \cite[10.4.5]{KatzGKM} and \cite[Lemma 4.9]{FKMSBilinearKS}, the representation of $I_\infty$ corresponding to $\Klc_n$ is isomorphic to
\[[x\mapsto x^n]_*\left(\Lc_{\chi_2}^{n+1}\otimes \Lc_{\psi(xn)}\right),\]
where $\chi_2$ is the character of order $2$ of $\overline\F_p^\times$ and $\psi(x)=e(\tr(x)/p)$. In other words, it is isomorphic to
\[\Ind_{I_{\infty,n}}^{I_\infty}\left(\Lc_{\chi_2}^{n+1}\otimes \Lc_{\psi(xn)}\right),\]
where $I_{\infty,n}$ is the unique subgroup of index $n$ in $I_\infty$ (see \cite[1.13]{KatzGKM}).

  \begin{equation*}
  \xymatrix@R=0.5cm{
    &k(Y)(Z,W)=k(Z,W)\ar@{-}[dl]_{k}\ar@{-}[dr]^{\Z/2}&\\
    k(Y)(Z)\ar@{-}[dr]_{\Z/2}&&k(Y)(W)\ar@{-}[dl]^{k}\\
    &k(Y)\ar@{-}[d]^{\Z/n}&\\
    &k(T)=k(Y^n)&
  }
\end{equation*}

The extension $k(Z,W)/k(T)$ is Galois, and we have a split exact sequence
\begin{equation*}
  \xymatrix@R=0.5cm@C=0.25cm{
    0\ar[r]&k\ar@{-}[d]_{\rotatebox{90}{$\cong$}}\ar[r]&G\ar@{=}[d]\ar[r]&\Z/2n\ar@{-}[d]_{\rotatebox{90}{$\cong$}}\ar[r]&0\\
    &\Gal(k(Z,W)/k(Z))&\Gal \left(k(Z,W)/k(T)\right)&\Gal(k(Y)(Z)/k(T)),&
  }  
\end{equation*}
so an isomorphism
\begin{eqnarray*}
  k\rtimes \Z/2n&\to&G=\Gal \left(k(Z,W)/k(T)\right)\\
  (a,i)&\mapsto&\sigma(a,i)
\end{eqnarray*}
  where $\sigma(a,i)$ is such that $W\mapsto \zeta_{2n}^{2i}(W+a)$ and $Z\mapsto \zeta_{2n}^{i}Z$. For every $(a,i)\in G$, there exists an element of $I_{\infty,n}$ extending $\sigma(a,i)$, that we will again denote by $\sigma(a,i)$.
  
We have $I_{\infty,n}=I_\infty\cap\pi_{1,q}^{(n)}$, where $\pi_{1,q}^{(n)}=\Gal(\closure{K(T)}/K(Y))$ is the subgroup of index $n$ in $\pi_{1,q}$. Indeed,
\begin{eqnarray*}
  I_\infty/(I_\infty\cap\pi_{1,q}^{(n)})&=&I_\infty\pi_{1,q}^{(n)}/\pi_{1,q}^{(n)}=\pi_{1,q}/\pi_{1,q}^{(n)}\\
  &\cong&\Gal(k(Y)/k(T))\cong\mu_n(k)\cong\Z/n.
\end{eqnarray*}
 Note that $(\sigma_i)_{1\le i\le n}$ is a complete reduced system of representatives of $I_\infty/I_{\infty,n}$, where we abbreviate $\sigma_i=\sigma(0,i)$.

By definition (or properties) of induced representations, a matrix form of the representation $I_\infty\to\GL_n(A)$ evaluated at $\sigma=\sigma(a_0,i_0)\in I_{\infty}$ is then
\[\left(\left(\Lc_{\chi_2}^{n+1}\otimes \Lc_{\psi(xn)}\right)(\sigma_{i,j}(\sigma))\delta_{\sigma_{i,j}(\sigma)\in I_{\infty,n}}\right)_{1\le i,j\le n,}\]
where $\sigma_{i,j}(\sigma)=\sigma_i^{-1}\sigma\sigma_j$. It remains to note that $\sigma_{i,j}(\sigma)\in I_{\infty,n}$ if and only if $2(i-j)\equiv 2i_0\pmod{2n}$, in which case:
\begin{itemize}
\item By definition of the Artin-Schreier representation,
  \[\Lc_{\psi(xn)}(\sigma_{i,j}(\sigma))=\psi((\sigma_{i,j}(\sigma)(W)-W)n),\]
  and $\sigma_{i,j}(\sigma)(W)-W=\zeta_{2n}^{2(j+i_0)}a_0$.
\item By definition of the Kummer representation, if $n$ is even,
  \[\Lc_{\chi_2}^{n+1}(\sigma_{i,j}(\sigma))=\chi_2(\zeta_n^{(j-i+i_0)/2})=(-1)^{\frac{j-i+i_0}{n}}.\]
\end{itemize}
Since the two representations are defined over $\F_\lambda$, the same holds for the isomorphism by the rational canonical form.
\end{proof}

\begin{remark}
  In particular, the image of the representation of $I_\infty$ contains an element conjugate to the (permutation, up to signs) matrix $m$ from Equation \eqref{eq:mu}. Note that $m$ has order $n$ (resp. $2n$) if $n$ is odd (resp. even).
\end{remark}

\subsubsection{Fields generated by traces}

The following will be useful to deal with subfield subgroups; it shows that we still recover conjugacy-invariant arithmetic information (the subfield generated by the traces of the Frobenius) in the geometric monodromy group.
\begin{proposition}\label{prop:fieldsTraces}
  Let $n\ge 2$, let $\lambda$ be an $\ell$-adic valuation on
  \[
      \Oc=
      \begin{cases}
        \Z[\zeta_p]&n\text{ odd or }p\equiv 1\pmod{4}\\
        \Z[\zeta_{4p}]&\text{otherwise, with }\ell\equiv 1\pmod{4}
      \end{cases}
    \]
   \textup{(}see \eqref{eq:fieldDef}\textup{)}, and let $\rho_n: \pi_{1,q}\to\GL_n(\F_\lambda)$ be the representation corresponding to the Kloosterman sheaf $\Klc_n$ of $\F_\lambda$-modules over $\P^1/\F_q$. Then
  \[\F_\ell \left(\Kl_{n,q}(a) : a\in\F_q^\times\right)=\F_\ell \left(\tr\rho_n(I_\infty)\right),\]
  with index $(f,n)$ in $\F_\lambda$.
\end{proposition}
\begin{proof}
    Under the hypotheses, $\F_\lambda=\F_\ell(\zeta_{p})$ (since $\zeta_{4}\in\F_\ell$ if $\ell\equiv 1\pmod{4}$) and $f:=[\F_\lambda:\F_\ell]=\ord(\ell\in\F_p^\times)$. Fisher \cite[Proposition 2.8]{Fisher95} showed that for $\Q(\zeta_p)$-valued Kloosterman sums,
    \[\Q\left(q^{\frac{n-1}{2}}\Kl_{n,q}(a) : a\in\F_q^\times\right)=\Q(\zeta_p)^{\Gal(\Q(\zeta_p)/\Q)[n]}.\]
    We proceed similarly to show that for $G=\Gal(\F_\ell(\zeta_p)/\F_\ell)$,
    \[
      \begin{rcases}
        L_1:=\F_\ell \left(\Kl_{n,q}(a) : a\in\F_q^\times\right)\\
        L_2:=\F_\ell \left(\tr\rho_n(I_\infty)\right)
      \end{rcases}
      =\F_\ell(\zeta_p)^{G[n]}.
    \]
    For $\sigma\in G$, let $u_\sigma\in\F_p^\times$ be such that $\sigma(\zeta_p)=\zeta_p^{u_\sigma}$, and note that for $a\in\F_q^\times$,
    \begin{eqnarray*}
      \sigma(\Kl_{n,q}(a))&=&\Kl_{n,q}(au_\sigma^n)\\
      \sigma \left(\tr\rho(\sigma(a,0))\right)&=&\tr\rho(\sigma(au_\sigma,0))=\sum_{i=1}^n e \left(\frac{n\tr_{\F_q/\F_p}(au_\sigma\zeta_n^i)}{p}\right),
    \end{eqnarray*}
    where $\sigma(a,0)$ is as defined in the proof of Proposition \ref{prop:monodromyInftyExplicit}. Hence, $G[n]\le\Gal(\F_\ell(\zeta_p)/L_i)$ for $i=1,2$.

    On the other hand, let us assume that $\sigma\in\Gal(\F_\ell(\zeta_p)/L_2)$. For every character $\Lambda: \F_q^\times\to\overline\F_\ell^\times$, we define
    \begin{eqnarray*}
      S_2(\Lambda)&=&\sum_{a\in \F_q^\times}\tr\rho(\sigma(a,0))\Lambda(a)\\
                &=&\sum_{i=1}^n\overline\Lambda(\zeta_n^i) G_n(\Lambda)=n\delta_{\Lambda\mid_{\mu_n}=1}G_n(\Lambda),
    \end{eqnarray*}
      where $G_n(\Lambda):=\sum_{a\in \F_q^\times}e \left(\frac{n\tr_{\F_q/\F_p}(a)}{p}\right)\Lambda(a)\neq 0$ since $G_n(\Lambda)G_{-n}(\overline\Lambda)=q\in\F_\ell^\times$. Then, since $\sigma\mid_{L_2}=\id$, we have $S_2(\Lambda)=\overline\Lambda(u_\sigma)S_2(\Lambda)$, which yields that $\Lambda(u_\sigma)=1$ whenever $\Lambda\mid_{\mu_n}=1$. Thus, $\sum_{\Lambda\in \widehat{\F_q^\times/\mu_n}}\Lambda(u_\sigma)=(q-1)/n$, so that $u_\sigma\in\mu_n$, i.e. $\sigma\in G[n]$.

      Finally, $\Gal(\F_\ell(\zeta_p)/L_1)\le \Gal(\F_\ell(\zeta_p)/L_2)$ by Chebotarev's density theorem (see \cite[I.2.2, Corollary 2a]{Serre89}).
      
    The claim on the index follows from $|G[n]|=|(\Z/f)[n]|=(n,f)$.
  \end{proof}
  \begin{equation*}
  \xymatrix{
    \Q(\zeta_{p})\ar@{-}[d]^{G[n]}_{(p-1,n)}&\Z[\zeta_{p}]\ar@{-}[dd]&\lf\ar@{-}[dd]&\F_\lf=\F_\ell(\zeta_p)\ar@{-}[d]_{(f,n)}^{G_\ell[n]}\\
    \Q(\Kl_{n,q}(a) : a\in\F_q)\ar@{-}[d]_{\frac{p-1}{(p-1,n)}}&&&\F_\ell(\Kl_{n,q}(a) : a\in\F_q)\ar@{-}[d]_{\frac{f}{(f,n)}}\\
    \Q&\Z&\ell\neq p&\F_\ell
  }
\end{equation*}

\section{Proof of Theorem \ref{thm:KlcIntegralMono} from Theorem \ref{thm:classification}}

Under the assumptions of Theorem \ref{thm:KlcIntegralMono}, assume that there exists a maximal (proper) subgroup $G_\geom(\Klc_n)\le H\le\SL_n(\F_\lambda)$ if $n$ is odd (resp. $\Sp_n(\F_\lambda)$ if $n$ is even). By Proposition \ref{prop:propertiesKl}, $H$ satisfies the hypotheses of Theorem \ref{thm:classification}.

It remains to show that the four cases of the conclusion of the latter are excluded:
\begin{itemize}[leftmargin=*]
\item[(1)] For $T=\SO_n(\F_\lambda)$ ($n$ odd), we proceed as in \cite[11.5.2]{KatzGKM}. We have
\[H=N_{\SL_n(\F_\lambda)}(T)\cong T\times\mu_n(\F_\lambda).\]
Over $\C$, this follows from the fact that $T$ contains no nontrivial scalars and that all automorphisms are inner by Propositions \ref{prop:OutLieGroup} and \ref{prop:AutDynkin}. In the finite case, we must take into account diagonal and field automorphisms (see Proposition \ref{prop:OutLieGroupFinite} below): the result is given in \cite[(2.6.2), Cor. 2.10.4, Prop. 2.10.6]{KleiLieb90}.

Let $\rho: \pi_{1,q}\to\GL_n(\F_\lambda)$ be the representation corresponding to $\Klc_n$. By \ref{item:KltotallyWildInftyb} of Proposition \ref{prop:propertiesKl}, considering the character $\rho(I_\infty)\to H\to \mu_n(\F_\lambda)$, we must have $\rho(P_\infty)\le T$, which contradicts \ref{item:Klpairingodd} of the same proposition. Indeed, $T$ preserves a nonsingular symmetric bilinear form.\label{page:preserveSymBilForm}
\end{itemize}

For \ref{item:classification3} and \ref{item:classification4}, where subfields appear, we use Proposition \ref{prop:fieldsTraces}: under the hypothesis $([\F_\lambda:\F_\ell], n)=1$, we have
\begin{equation}
  \label{eq:fieldsTraces}
  \F_\lambda=\F_\ell(\tr\rho_n(I_\infty))\le\F_\ell(\tr(G_\geom)).
\end{equation}

\begin{lemma}\label{lemma:normalizers}
  For $L/k$ an extension of finite fields and $G\in\{\SL_n,\Sp_n\}$,
  \[[n]\No_{G(L)}(G(k))\subset G(k),\]
  where $[n]: G(L)\to G(L)$ is defined by $g\mapsto g^n$.
\end{lemma}
\begin{proof}
  Let $\sigma\in\Aut_k(L)$ be the Frobenius. If $g\in\No_{G(L)}(G(k))$, we have $ghg^{-1}\in G(k)$ for all $h\in G(k)$, i.e. $\sigma(ghg^{-1})=ghg^{-1}$ (applying $\sigma$ entry-wise), so $\sigma(g)h\sigma(g)^{-1}=ghg^{-1}$, which shows that $\sigma(g)g^{-1}\in C_{G(L)}(G(k))$. Recall that $\SL_n(k)$ is generated by elementary matrices $e_{i,j}$ ($i\neq j$) and that $\Sp_n(k)$ (with the usual form) contains the elementary matrices $e_{i,\sigma(i)}$ where $\sigma(i)=i+n-1\pmod{2n}$. This yields $C_{G(L)}(G(k))\le Z(G(L))$. Therefore, $\sigma(g)=\lambda_gg$ with $\lambda_g\in\mu_n(L)$. Since $\sigma(g^n)=\sigma(g)^n=g^n$, we get that $g^n\in G(k)$. See also \cite[4.5.3--4]{KleiLieb90}.
\end{proof}

\begin{itemize}[leftmargin=*]
  \item[(2)] If $H=N_{G(\F_\lambda)}(G(\F'))$ with $G\in\{\SL_n,\Sp_n\}$, then Lemma \ref{lemma:normalizers} shows that
  \[\tr\rho(\sigma(an,0))\in\F'\]
  for every $a\in \F_q$. Since $(n,p)=1$, it follows that $\tr\rho(I_\infty)\subset\F'$, which contradicts \eqref{eq:fieldsTraces} since $\F'$ is a proper subfield of $\F_\lambda$.
\item[(3)] Finally, we use a combination of the techniques used in \ref{item:classification1}--\ref{item:classification3} to handle the case of $H=N_{\SL_n(\F_\lambda)}(\SU_n(\F'))$. By \cite[Proposition 4.8.5]{KleiLieb90},
  \begin{eqnarray*}
    H=\mu_n(\F_\lambda)\SU_n(\F')&\cong& \frac{\SU_n(\F')\times\mu_n(\F_\lambda)}{\SU_n(\F')\cap\mu_n(\F_\lambda)}\\
    &=&\frac{\SU_n(\F')\times\mu_n(\F_\lambda)}{\mu_{(n,1+|\F'|)}(\F_\lambda)}.
  \end{eqnarray*}
  By Proposition \ref{prop:propertiesKl} \ref{item:KltotallyWildInftyb} applied to the representation
  \[H\to\mu_{n/(n,1+|\F'|)}(\F_\lambda)\]
  restricted to $\rho(I_\infty)$, we have $\rho(P_\infty)\le\SU_n(\F')$. In other words, $P_\infty$ leaves invariant the sesquilinear form on $(\F_\lambda)^n$ associated to the involution $\sigma\in\Aut(\F_\lambda)$, $x\mapsto x^{|\F'|}$.

  As in \ref{item:Klpairingodd} of Proposition \ref{prop:propertiesKl} (and its proof in \cite[4.1.5--4.1.8]{KatzGKM}), this yields an isomorphism of $P_\infty$-representations between $\rho$ and $\sigma(D(\rho))$. Equivalently, there exists $A\in\GL_n(\F_\lambda)$ such that
  \[A\rho(g)A^{-1}=\sigma\left(\rho(g^{-1})^t\right)\text{ for all }g\in P_\infty,\]
  so $A\rho(P_\infty)A^{-1}\le\GL_n(\F')$, which contradicts again \eqref{eq:fieldsTraces}.
\end{itemize}
This concludes the proof of Theorem \ref{thm:KlcIntegralMono} assuming Theorem \ref{thm:classification}. \qed

\begin{remarks}
  The conditions $p\nmid n$ and $([\F_\lambda:\F_\ell],n)=1$ in Theorem \ref{thm:KlcIntegralMono} are required to exclude subfield subgroups. In particular:
  \begin{itemize}
  \item If $p\mid n$, the elements produced with Proposition \ref{prop:monodromyInftyExplicit} lie in $\SL_n(\F_\ell)$.
  \item If $e=([\F_\lambda:\F_\ell],n)>1$, Proposition \ref{prop:fieldsTraces} shows that we cannot exclude that $G_\geom$ lies in $N_{G(\F_\lambda)}(G(\F'))$, for $\F'<\F_\lambda$ a proper subfield of index $e$ (where $G=\SL_n$ if $n$ is odd, $\Sp_n$ otherwise).
  \end{itemize}
  Moreover, in the case $p\equiv 3\pmod{4}$ and $\ell\equiv 2\pmod{3}$, we cannot exclude that $G_\geom$ is defined over $\F_\ell(\zeta_p)<\F_\ell(\zeta_{4p})=\F_\lambda$. Indeed, $\F_\lambda$ contains $\zeta_4$ because of the Tate twist, which is trivial on the geometric fundamental group.

  In \cite{Hall08} (and applications in \cite{KowLS06} and \cite{KowLargeSieve08}), subfield subgroups do not enter the picture because only sheaves of $\Z_\ell$-modules are considered, as opposed to reductions of sheaves of $\Z[\zeta_{4p}]_\lambda$-modules for Kloosterman sheaves.
\end{remarks}

\section{Proof of the classification Theorem \ref{thm:classification}}

\subsection{Reminder on the classification of maximal subgroups of classical groups}\label{subsec:reviewClassificationMax}
As already mentioned, an essential input to prove Theorem \ref{thm:classification} is the classification of maximal subgroups of classical groups, which originated from the work of Aschbacher \cite{Asch84}, and was then expanded by Kleidman-Liebeck \cite{KleiLieb90}. Another proof was given by Liebeck-Seitz \cite{LiebSei98} through an analogous result over an algebraically closed field and using descent (as in the treatment of finite groups of Lie type with Steinberg isomorphisms). A good exposition of these results can be found in \cite[II.18.1, III.27--28]{TesMal11}.\\

In the remaining of this section, we survey this classification in a way adapted to our needs.

  \begin{theorem}[{\cite{Asch84}, \cite{LiebSei98}}]\label{thm:classifMaxSubgroups}
  Let $\F_\lambda$ be a finite field of odd characteristic $\ell$ and for $n\ge 2$, let $G=\SL_n(\F_\lambda)$ or $G=\Sp_n(\F_\lambda)$ \textup{(}$n$ even\textup{)}. We denote by $\pi:\GL_n(\overline\F_\lambda)\to\PGL_n(\overline\F_\lambda)$ the projection. If $H$ is a maximal (proper) subgroup of $G$, then either:
  \begin{enumerate}
  \item\label{item:classifMaxSubgroups1} $H$ belongs to one of the classes $\Cc_1, \dots, \Cc_7$ described below, or
  \item\label{item:classifMaxSubgroups2} $\pi(H)$ is \emph{almost simple}: there exists a simple group $S$ such that
    \[S\cong\Inn(S)\normal \pi(H)\le\Aut(S).\]
    Moreover, $H$ admits a unique normal subgroup $T$ such that $\pi(T)=S$, and the action of $T\normal H\le \SL_n(\F_\lambda)$ on $\overline\F_\lambda^n$ is absolutely irreducible. If $G=\SL_n(\F_\lambda)$, then $T$ preserves no nondegenerate bilinear or unitary form on $\F_\lambda^n$.
  \end{enumerate}
\end{theorem}
\begin{proof}
  This is a combination of Theorems 1 and 2 from \cite{LiebSei98}.
\end{proof}

We now recall the definition of the classes $\Cc_1, \dots, \Cc_7$, along with some convenient properties we will use.\\

Let $V=\F_\lambda^n$, $\Vb=\overline{\F}_\lambda^n$, and $\Frob\in\Gal(\overline\F_\lambda/\F_\lambda)$ be the \textit{arithmetic} Frobenius $x\mapsto x^{|\F_\lambda|}$. We will write $\Gb=\SL_n(\overline\F_\lambda)$ (resp. $\Sp_n(\overline\F_\lambda)$). Let $\beta$ be the zero bilinear form on $V$ if $G=\SL_n(\F_\lambda)$ or the symplectic form associated to $G$ if $G=\Sp_n(\F_\lambda)$. The classes appearing in \ref{item:classifMaxSubgroups1} of Theorem \ref{thm:classifMaxSubgroups} are the following:
\begin{itemize}[leftmargin=*]
\item Class $\Cc_1$ (subspace stabilizers):
  \[H=\Stab_G(W)\] with $0\neq W\lneq V$ totally singular or nondegenerate with respect to $\beta$. Note that $W\le V$ is an $H$-submodule, so this case does not arise if $H$ acts on $V$ irreducibly.
\item Class $\Cc_2$ (stabilizers of orthogonal decompositions):
  \begin{eqnarray*}
    \Vb&=&\Vb_1\perp\dots\perp \Vb_t\\
    &&(t\ge 2, \  \text{all the }\Vb_i \ \text{isometric}, \ n=\dim(\Vb_1)t),\\
    \bs M&=&\Stab_{\Gb}(\Vb_1\perp\dots\perp \Vb_t),\\
    H&\le&\bs M^{\Frob}.
  \end{eqnarray*}
  In other words, the elements of $\bs M$ are the $g\in\Gb$ such that there exists a permutation $\sigma\in \Sf_t$ with $g\Vb_i=\Vb_{\sigma(i)}$ for all $1\le i\le t$.
\item Class $\Cc_3$ (stabilizers of totally singular decompositions): if $G=\Sp_n(\F_\lambda)$,
  \begin{eqnarray*}
    \Vb&=&\Vb_1\oplus \Vb_2 \hspace{0.3cm} (\Vb_i\text{ maximal totally isotropic: }\beta|_{\Vb_i}=0),\\
    \bs M&=&\Stab_{\Gb}(\Vb_1\oplus \Vb_2),\\
    H&\le&\bs M^{\Frob}.
  \end{eqnarray*}
  In other words, the elements of $M$ are the $g\in\Gb$ such that there exists a permutation $\sigma\in \Sf_2$ with $g\Vb_i=\Vb_{\sigma(i)}$ for $i=1,2$. In particular, $\dim(\Vb_1)=\dim(\Vb_2)=n/2$.
\item Class $\Cc_4$ (stabilizers of tensor product decompositions):
  \begin{eqnarray*}
    \Vb&=&\Vb_1\otimes\dots\otimes \Vb_t \hspace{0.3cm} (\dim \Vb_i\ge 2, \ t\ge 2),\\
    \bs L&=&\Gb(\Vb_1)\times\dots\times \Gb(\Vb_t)\text{, acting on }\Vb\text{ by tensor product},\\
    \bs M&=&N_{\GL_n(\overline\F_\lambda)}(\bs L)\cap \Gb,\\
    H&\le&\bs M^{\Frob},
  \end{eqnarray*}
  with $t=2$ if the $\bs V_i$ are not mutually isomorphic, where we write $\Gb(\Vb_i)$ for the classical group of type $\Gb$ on the vector space $\Vb_i$. Note that $n=\dim(\Vb_1)\dots\dim(\Vb_t)$, and $n=\dim(\Vb_1)^t$ if the $\Vb_i$ are mutually isomorphic. We have
\[\pi\left(N_{\GL_n(\overline\F_\lambda)}(\bs L)\cap \Gb\right)\le N_{\PGL_n(\overline\F_\lambda)}(\pi(\bs L))\cap \pi(\Gb)=\bs N.\]
Since $\pi(\Gb)$ has trivial center, there is a morphism from $\bs N$ to
\begin{eqnarray*}
  \Aut(\pi(\bs L))&=&\Aut\left(\pi(\Gb(\Vb_1))\times\dots\times\pi(\Gb(\Vb_t))\right)\\
  &\cong&
  \begin{cases}
    \Aut\left(\pi(\Gb(\Vb_1))\right)\times \Aut\left(\pi(\Gb(\Vb_2))\right)&:t=2\\
    \Aut\left(\pi(\Gb(\Vb_1))\right)\wr \Sf_t&:t>2
  \end{cases}\\
  &\to&  \begin{cases}
    \Out\left(\pi(\Gb(\Vb_1))\right)\times \Out\left(\pi(\Gb(\Vb_2))\right)&:t=2\\
    \Out\left(\pi(\Gb(\Vb_1))\right)\wr \Sf_t&:t>2,
  \end{cases}
\end{eqnarray*}
with kernel isomorphic to $\pi(\bs L)\cap\pi(\Gb)$. The isomorphism on the second line follows from:
\begin{lemma}
  Let $G_1, \dots, G_t$ be nonabelian simple groups. Then $\Aut(G_1\times\dots\times G_t)$ is isomorphic to
  \[\Aut(G_1)\times\dots\times\Aut(G_t)\]
  if there is no isomorphism among the $G_i$, respectively $\Aut(G_1)\wr \Sf_t$ if the $G_i$ are mutually isomorphic.
\end{lemma}
\begin{proof}
  Proceed as in the second paragraph of the proof of \cite[3.3.20]{Rob96}.
\end{proof}
\item Class $\Cc_5$ (symplectic-type $r$-subgroups):
  \[\pi(H)\cong
  \begin{cases}
    \Z/r^{2m}.\Sp_{2m}(\F_r)&:G=\SL_n(\F_{r^m})\\
    \Z/2^{2m}.\GO^-_{2m}(\F_2)&:G=\Sp_n(\F_{2^m})
  \end{cases}
\]
with $n=r^m$, $r\neq \ell$ prime, $\ell\equiv 1\pmod{r(2,r)}$. Here, we have only given the classification of the subgroups that arise in the class; for more details about the latter, see \cite[Section 4.6]{KleiLieb90}.
\item Class $\Cc_6$ (normalizers of classical groups): For $G=\SL_n(\F_\lambda)$ with $n$ odd (since $\ell\neq 2$, this class does not arise in the symplectic case) and $\F'\le\F_\lambda$ a subfield such that $|\F'|=|\F_\lambda|^{1/2}$,
  \[H=N_{G}(\SO_n(\F_\lambda))\text{ or } N_{G}(\SU_n(\F')).\]
\item Class $\Cc_7$ (subfield subgroups): For $\F'\le\F_\lambda$ of prime index,
  \[H=N_{G}(G(\F')).\]
\end{itemize}

Note that the unitary cases of classes $\Cc_6$ and $\Cc_7$ do not arise if $\F_\lambda=\F_\ell$.

\begin{remark}
  Similar results hold for other classical groups and the description of the classes $\Cc_i$ can be made more explicit (see \cite{LiebSei98}).
\end{remark}

\subsection{Review on automorphism groups}
We briefly recall the following results about automorphisms of Lie algebras/Lie groups/finite groups of Lie type that will be useful several times, in particular to handle class $\Cc_4$ and case \ref{item:classifMaxSubgroups2} of Theorem \ref{thm:classifMaxSubgroups}:

\begin{proposition}\label{prop:OutLieGroup}
  If $G$ is a simple Lie algebra (resp. a simply connected simple Lie group) over an algebraically closed field, there is an isomorphism between $\Out(G)$ and the group $\mathrm{Graph}(G)$ of graph automorphisms of the corresponding Dynkin diagram.
\end{proposition}
\begin{proof}
  This can be found in \cite[Chapter 16.5]{Hum80} and \cite[Proposition D.40]{FulHar91}.
\end{proof}
In the finite case, this becomes:
\begin{proposition}\label{prop:OutLieGroupFinite}
  If $G$ is a finite simple group of Lie type defined over a finite field $k$, every automorphism can be written as the product of an inner, graph, diagonal, and field automorphism. More precisely,
  \[\Out(G)\cong \left(\mathrm{Diag}(G)\Aut(k)\right).\mathrm{Graph}(G),\]
  where $\mathrm{Diag}(G)$ is the group of diagonal automorphisms.
\end{proposition}
\begin{proof}
  See \cite[4.237]{Gor82} and \cite[Theorem 12.5.1]{Carter72}.
\end{proof}

\begin{proposition}\label{prop:AutDynkin}
  The automorphism group of a connected Dynkin diagram is $\Z/2$ for $A_n$, $D_n$ \textup{(}$n>1$\textup{)} and $E_6$, $\Sf_3$ for $D_4$, and trivial otherwise.
\end{proposition}

\subsection{Strategy to prove Theorem \ref{thm:classification}}

As already mentioned, we apply Theorem \ref{thm:classifMaxSubgroups} to the situation of Theorem \ref{thm:classification}. Let us first sketch the general approach.

\subsubsection{Geometric subgroups}

First, we exclude subgroups from classes $\Cc_1,\dots,\Cc_5$:

\begin{proposition}\label{prop:geometricSubgroups}
    Let $H\in\bigcup_{i=1}^5 \Cc_i$ be a maximal proper subgroup as in Theorem \ref{thm:classifMaxSubgroups}. If $H$ acts irreducibly on $V$ and contains a unipotent element $u$ with a single Jordan block, then $\ell\ll_n 1$.
\end{proposition}
  
Classes $\Cc_6$ and $\Cc_7$ will remain in the conclusion of Theorem \ref{thm:classification}.

\subsubsection{Almost-simple subgroups}

On the other hand, if $H$ satisfies \ref{item:classifMaxSubgroups2} of Theorem \ref{thm:classifMaxSubgroups}, there exists a nonabelian simple group $S$ such that 
  \[S=\pi(T)\normal \tilde H=\pi(H)\le\Aut(S)\]
  with $T\normal H\le\SL_n(\F_\lambda)$ and $T$ acting irreducibly on $\overline\F_\lambda^n$, for $\pi: \SL_n\to\PSL_n$ the projection.\\

  The first step is to reduce to the case where $S$ is a group of Lie type in characteristic $\ell$, applying the following with the regular unipotent element contained in $H$:
  
  \begin{proposition}\label{prop:reduceLieCharell}
    In the above setting, if $H$ contains a nonscalar element of order $\ge \ell$, then either $\ell\ll_n 1$ or $S$ is a group of Lie type in characteristic $\ell$.
  \end{proposition}

  Our first approach relied on the classification of finite simple groups, but this is also a particular case of a recent powerful theorem of Larsen and Pink, independent from the classification. We will give both arguments.\\

  For groups of Lie type in characteristic $\ell$, we have:

  \begin{proposition}\label{prop:LieCharell}
    In the above setting, if $S=\pi(T)$ is a simple group of Lie type in characteristic $\ell\gg_n 1$ and $H$ contains a regular unipotent element, then either:
    \begin{enumerate}
    \item $T=\SL_n(\F_\lambda)$ or $T=\Sp_n(\F_\lambda)$, or
    \item The inclusion $T\to\SL_n(\F_\lambda)$ fixes a nondegenerate bilinear or unitary form on $\F_\lambda^n$, and if $n\ge 4$, the image of $T$ is not contained in $\Sp_n(\F_\lambda)$.
    \end{enumerate}
  \end{proposition}
  To prove this, we first show that the degree of the field of definition must be small, according to results of Liebeck on the minimal dimension of faithful irreducible modular representations of simple groups of Lie type. This implies that the group of outer automorphisms is small, and then that $S$ must contain the regular unipotent element of $\tilde H$ for $\ell\gg_n 1$. Over an algebraically closed field, the irreducible representations of a semisimple algebraic group with central kernel whose images contain an element with a single Jordan block are classified by a result of Suprunenko. The absolute irreducibility of the action of $T$ allows to descend to finite groups of Lie type by a result of Seitz-Testerman.
  
\begin{remark}
  The strategy to exclude alternating groups and groups of Lie type in cross-characteristic in the almost simple case of the characterization is quite standard (see e.g. \cite[Chapter 28]{TesMal11}). The results of Liebeck mentioned are notably used in \cite{KantLub90} to determine the probability that two random elements of $\PSL_n(\F_\ell)$ are generators, in the case $\ell\le 9$.
\end{remark}

\subsubsection{Conclusion}

Since Theorem \ref{thm:classifMaxSubgroups} \ref{item:classifMaxSubgroups2} excludes that $T$ fixes a nondegenerate bilinear or unitary form on $\F_\lambda^n$ when $n$ is odd, the only possibilities that remain are $T=\SL_n(\F_\lambda)$ when $n$ is odd and $T=\Sp_n(\F_\lambda)$ when $n$ is even. The classification theorem \ref{thm:classification} follows.

\begin{remark}
  According to \cite{Lieb85}, building on Theorem \ref{thm:classifMaxSubgroups} and the classification of finite simple groups, a maximal subgroup $H$ of a classical group over a finite field $\F_\lambda$ of characteristic $\ell$ satisfies one of the following:
  \begin{itemize}
  \item $H$ belongs to one of the families $\Cc_1,\dots,\Cc_7$.
  \item $H$ is $\Alt(c)$ or $\Sf_c$ with $c\in\{n+1,n+2\}$, and $H\to\GL_n(\F_\lambda)$ is the representation of minimal dimension.
  \item $|H|<|\F_\lambda|^{3n}$.
  \end{itemize}
  As we just mentioned, it is relatively easy to exclude the first two families by using the presence of a regular unipotent element or the growth of $H$ when $\ell\to\infty$. Since $|\SL_n(\F_\lambda)|=\Theta_n(|\F_\lambda|^{n^2-1})$ and $|\Sp_n(\F_\lambda)|=\Theta_n(|\F_\lambda|^{n(2n+1)})$, the result of Liebeck shows that in the remaining cases $H$ is quite small. In other words, we would only need to show that the monodromy group is ``moderately big'' to show that it is the full classical group expected.
\end{remark}
\subsection{Excluding members of classes $\Cc_1,\dots,\Cc_5$ (Proposition \ref{prop:geometricSubgroups})}

\begin{lemma}\label{lemma:elementOrderell}
  If $\ell\gg_n 1$, a unipotent element with a single Jordan block in $\GL_n(\F_\ell)$ has order $\ell$.
\end{lemma}
\begin{proof}
  The element is conjugate to the Jordan block $u$ satisfying $(u^k)_{ij}=\delta_{i\le j\le k}\binom{k}{j-i}$ for $1\le i,j\le n$, so it has order $\ell^{\ceil{\log_\ell n}}$ by Lucas' theorem. In particular, this is $\ell$ if $\ell\gg_n 1$.
\end{proof}

We prove Proposition \ref{prop:geometricSubgroups}, assuming throughout that $\ell\gg_n 1$ to apply Lemma \ref{lemma:elementOrderell}.


\subsubsection{Class $\Cc_1$} The first class is excluded by the assumption that $H$ acts irreducibly on $V$.
\subsubsection{Class $\Cc_2$} By assumption, there exists a permutation $\sigma\in \Sf_t$ such that $u\Vb_i=\Vb_{\sigma(i)}$ for all $i$. Note that there is at most one $\Vb_i$ which is $u$-stable, since the $\Vb_i$ are disjoint with equal dimension, and $u$ has exactly only invariant subspace of each dimension $0\le d\le n$. In particular, $\sigma$ has at most one fixed point. Write $\sigma=\sigma_1\dots\sigma_k$ where $\sigma_1,\dots,\sigma_k$ are disjoint cycles, with $\sigma_j$ of length $2\le s_j\le n$. Since $u$ has order $\ell$, we have
\[\Vb_i=u^\ell \Vb_i=\Vb_{\sigma^\ell(i)},\]
so $\sigma^\ell=\id$ and either
\begin{itemize}
\item $\sigma=\id$, which implies that all $\Vb_i$ are $u$-stable, a contradiction.
\item $\ord(\sigma)=\lcm(s_1,\dots,s_k)=\ell$. Hence $s_j=\ell$ for all $j$, thus $k\ell=n-|\fix(\sigma)|$, i.e. $\ell\mid n$ or $\ell\mid n-1$. This can be excluded if $\ell>n$.
\end{itemize}
\subsubsection{Class $\Cc_3$} (for $G=\Sp_n(\F_\lambda)$). This is excluded in the same way as class $\Cc_2$.
\subsubsection{Class $\Cc_4$} Consider the morphism
\[N_{\PGL(\overline\F_\lambda)}(\pi(\bs L))\cap \pi(\Gb)\to \begin{cases}
    \Out\left(\pi(\Gb(\bs V_1))\right)\times \Out\left(\pi(\Gb(\bs V_2))\right)&:t=2\\
    \Out\left(\pi(\Gb(\bs V_1))\right)\wr \Sf_t&:t>2
  \end{cases}\]
with kernel $\pi(\bs L)\cap\pi(\Gb)$. If $\pi(u)\not\in\pi(\bs L)$, then the order of the image of $\pi(u)$ is
\[\ell\mid\left|\Out\left(\pi(\Gb(\Vb_i))\right)\right|^tt!=2^tt!\]
for some $i\in\{1,2\}$, by Propositions \ref{prop:OutLieGroup} and \ref{prop:AutDynkin}. Thus $\ell\ll_n 1$ since $t\le n$. On the other hand, elements in $\pi(\bs L)$ have at least two Jordan blocks\footnote{Since the center of $\GL_n(\overline\F_\lambda)$ is the group of scalar matrices, it makes sense to speak of the number of Jordan blocks of an element in $\PGL_n(\overline\F_\lambda)$.}, which also rules out the case $\pi(u)\in\pi(\bs L)$. Indeed, if $g_1\in \Gb(\Vb_1)$ and $g_2\in \Gb(\Vb_2)$ are two Jordan blocks, then $g_1\otimes g_2$ fixes the linearly independent vectors $v_{1,1}\otimes v_{2,1}$ and $v_{1,2}\otimes v_{2,1}-v_{1,1}\otimes v_{2,2}$, where $v_{i,1},v_{i,2}$ are the first two elements of the standard basis of $\Vb_i$ ($i=1,2$).

\subsubsection{Class $\Cc_5$} Consider the morphism
\[H\to \pi(H)\to \pi(H)/(\Z/r^m)\cong\Sp_{2m}(\F_r).\]
Since $u$ has order $\ell\neq r$, the image of $u$ in $\Sp_{2m}(\F_r)$ still has order $\ell$, hence
\[\ell\mid |\Sp_{2m}(\F_r)|=r^{m^2}(r^2-1)(r^4-1)\dots (r^{2m}-1)\le r^{m(2m+1)},\]
which implies that $\ell\ll_n 1$ because $n=r^m$. \qed
\subsection{Excluding almost simple groups (Propositions \ref{prop:reduceLieCharell} and \ref{prop:LieCharell})}
Let us now assume there exists a nonabelian simple group $S$ such that
\[S=\pi(T)\normal \pi(H)=\tilde H\le\Aut(S),\]
with $T\normal H\le\SL_n(\F_\lambda)$ and $T$ acting irreducibly on $\overline\F_\lambda^n$.
\subsubsection{Reduction to groups of Lie type in characteristic $\ell$}

We prove Proposition \ref{prop:reduceLieCharell}. Since $u\in H$ has order $\ge \ell$ and is not scalar, we have $|H|, |\tilde H|\ge \ell$. The result then follows from:
  
\begin{theorem*}\label{thm:reduceLieCharell2}
  Let $S\le\PGL_n(\F_\lambda)$ be a simple group. Then either $|S|\ll_n 1$, or $S$ is a group of Lie type in characteristic $\ell$.
\end{theorem*}

This is a direct consequence of \cite[Theorem 0.2]{LarsPink11} (see \cite[Theorem 0.3]{LarsPink11}), exploiting the theory of algebraic groups.

Let us nonetheless show how it also follows from the classification of finite simple groups (\cite[no. 1, p. 6]{Gor94}). According to the latter, it suffices to prove:
\begin{lemma}
  If $S$ is sporadic, alternating or of Lie type in characteristic coprime to $\ell$, then $\ell\ll_n 1$.
\end{lemma}
\begin{proof}
  First note that we have $|\Aut(S)|\ll_n 1$:
  \begin{itemize}
  \item If $S$ is sporadic, this is clear.
  \item If $S=\Alt(m)$ (with $m\ge 5$), then Wagner \cite[Theorem 1.1]{Wag77} showed that the dimension of a faithful modular representation of $S$ is at least $m-2$. Since $S\le \tilde H\le\PSL_n(\F_\lambda)$, it follows that $m\le n+2$, so $|\Aut(S)|\ll_n 1$.
  \item If $S$ is of Lie type of rank $l$ over a field $\F_r$ of characteristic distinct from $\ell$, then the main theorem of Landazuri-Seitz \cite{LandSei74} shows that $r,l\ll_n 1$, so that $|\Aut(S)|\ll_n 1$.
  \end{itemize}
  Hence $\ell\le |\tilde H|\le|\Aut(S)|\ll_n 1$.
\end{proof}

\subsubsection{Groups of Lie type in characteristic $\ell$}

We finally prove Proposition \ref{prop:LieCharell}. Let us assume that $S$ is a group of Lie type of rank $l$ over $\F_r$, with $r=\ell^a$. We continue to assume that $\ell\gg_n 1$ so that the regular unipotent element $u$ has order $\ell$.

The first difficulty to overcome is that we do not know a priori whether $S$ itself contains a regular unipotent element. However, we can show:
\begin{proposition}\label{prop:uinS}
  If $\ell\gg_n 1$, then $\pi(u)\in S$ and $T$ contains as well an element with a single Jordan block.
\end{proposition}

We start by proving this in the following paragraphs. Recall that we have an exact sequence
\[1\to S\cong\Inn(S)\to\Aut(S)\to\Out(S)\to 1,\]
an inclusion $S\le \tilde H\le\Aut(S)$, and $\pi(u)\in \tilde H$ of order $\ell$. If $\pi(u)\not\in S$, then its image in $\Out(S)$ has order $\ell$ and so $\ell$ divides $|\Out(S)|$. Thus, it suffices to show that $|\Out(S)|\ll_n 1$ to rule out this possibility.
\begin{lemma}\label{lemma:sizeOutS}
  We have
  \[|\Out(S)|=Na\]
  with $N\in\{1,2,6,8,12\}$, unless
\begin{itemize}
\item $S=A_l(r)$ with $l\ge 3$ odd, where we have $|\Out(S)|=2a(l+1,r-1)$, or
\item $S=\prescript{2}{}{A}_l(r)$ with $l\ge 3$ odd, where we have $|\Out(S)|=2a(l+1,r+1)$.
\end{itemize}
\end{lemma}
\begin{proof}
  See Propositions \ref{prop:OutLieGroupFinite}, \ref{prop:AutDynkin} and Table \ref{table:OutS}: there are $a$ field automorphisms, $1,2$ or $3$ graph automorphisms, and less that $4$ diagonal automorphisms, except for $A_l$ and $\prescript{2}{}{A}_l$ which have respectively $(l+1,r-1)$ and $(l+1,r+1)$ diagonal automorphisms.
\end{proof}
Letting $m(S)$ be the minimal dimension of a faithful irreducible projective representation of $S$ over an algebraically closed field of characteristic $\ell$, the following result lets us bound the rank of $S$ and the degree of its defining field:
\begin{lemma}\label{lemma:boundLiebeck}
  We have
  \[l\le m(S)\le n^{([\F_\lambda:\F_\ell],a)/a},\]
  whence $l,a\ll_n 1$. In particular, for $n$ fixed, there is only a finite number of possibilities for $S$.
\end{lemma}
\begin{proof}
  The bounds follow from \cite[(2.1)--(2.2)]{Lieb85} and the fact that $S\le\tilde H\le\PSL_n(\F_\lambda)$. Since $m(S)\ge 2$ (see Table \ref{table:m(S)}), we have $a\le \frac{\log{n}}{[\F_\lambda:\F_\ell]\log{2}}\ll_n 1$, so that $l,a\ll_n 1$.
\end{proof}
\begin{table}
  \centering
  \begin{tabular}{c|c}
    $S$&$m(S)$\\ \hline
    $A_l, \prescript{2}{}{A}_l$&$l+1\ge 2$\\
    $B_l$&$2l+1\ge 5$\\
    $C_l$&$2l\ge 6$\\
    $D_l, \prescript{2}{}{D}_l$&$2l\ge 6$\\
    $\prescript{3}{}{D}_4$&$8$\\
    $G_2$&$7$\\
    $F_4$&$26$\\
    $E_6,\prescript{2}{}{E}_6$&$27$\\
    $E_7$&$56$\\
    $E_8$&$248$
  \end{tabular}
  \caption{Minimal dimension of a faithful irreducible projective representation of a simple group of Lie type over an algebraically closed field in the same characteristic $p>3$, according to \cite[Table 2]{Lieb85} or \cite[Table 1]{KantLub90}.}
  \label{table:m(S)}
\end{table}

By Lemmas \ref{lemma:sizeOutS} and \ref{lemma:boundLiebeck},
\[|\Out(S)|\le 12a(l+1)\ll_n 1,\]
which concludes the proof of Proposition \ref{prop:uinS}.\qed\\

Back to Proposition \ref{prop:LieCharell}, the next difficulty is that we do not know whether the action of $T$ on $\F_\lambda^n$ as a group of Lie type is the action induced by the inclusion $T\le H\le\SL_n(\F_\lambda)$. However, thanks to Proposition \ref{prop:uinS} and the irreducibility of the action of $T$ on $\Vb$, the representation can be determined by the following:
\begin{theorem}\label{thm:Supr}
  Let $T$ be a finite group of Lie type in characteristic $\ell\ge 5$, of simply connected type, with a faithful absolutely irreducible representation $\varphi: T\to\SL_n(\F_\lambda)$. Assume that $\varphi(T)$ contains an element with a single Jordan block. Then either (up to conjugacy):
  \begin{enumerate}
  \item $T=\SL_n(\F_\lambda)$, $\Spin_n(\F_\lambda)$ for $n$ odd or $\Sp_n(\F_\lambda)$ for $n$ even, with the standard embedding in $\SL_n(\F_\lambda)$.
  \item $T=\SU_n(\F')$ with the standard embedding in $\SL_n(\F_\lambda)$, for $\F'\le\F_\lambda$ a subfield such that $|\F'|=|\F_\lambda|^{1/2}$.
  \item $T=G_2(\F_\lambda)$ and $n=7$, with $\varphi$ the unique 7-dimensional irreducible representation.
  \end{enumerate}
\end{theorem}

\begin{remark}
  This is to be compared with the fact that the only nontrivial irreducible complex representations of $\SL_n(\C)$ of dimension $\le n$ are the standard representation and its dual.
\end{remark}

  Theorem \ref{thm:Supr} is a version of \cite[Theorem (1.9)]{Sup95} for finite groups of Lie type\footnote{Suprunenko notes in the article that the results could be ``\emph{easily transferred to irreducible $\F_\ell$-representations of finite Chevalley groups over fields of characteristic $\ell$}.'' It seems however that we need to restrict to absolutely irreducible representations.}. To prove this variant, we use the lifting theorem of Seitz and Testerman:
  \begin{theorem}[{\cite[Theorem 1, case $G=\SL_n$]{SeiTes90}\footnote{See also \cite[Section 29.2]{TesMal11}.}}]\label{thm:SeiTes}
    Let $\bs H$ be a simple algebraic group over $\overline\F_\lambda$, with a Steinberg endomorphism $F: \bs H\to \bs H$, and $X=[\bs H^{F},\bs H^{F}]$ perfect. If $\varphi: X\to \SL_n(\F_\lambda)$ is a morphism such that $\varphi(X)$ lies in no proper $F$-stable parabolic subgroup of $\SL_n(\overline\F_\lambda)$, then $\varphi$ can be extended to a morphism of algebraic groups $\bs\varphi: \bs H\to \SL_n(\overline\F_\lambda)$ with $\bs\varphi\mid_X=\varphi$.
\end{theorem}
\begin{proof}[Proof of Theorem \ref{thm:Supr}]
  By hypothesis, $T=\Tb^F$ for $\Tb$ a simple algebraic group over $\overline\F_\ell$ and $F: \Tb\to\Tb$ a Steinberg endomorphism. We consider the absolutely irreducible representation $\varphi: T\to\SL_n(\F_\lambda)$. By irreducibility, the image of $T$ is not contained in a proper parabolic subgroup of $\SL_n(\F_\lambda)$. Theorem \ref{thm:SeiTes} thus shows the existence of a morphism $\bs\varphi: \Tb\to\SL_n(\overline\F_\ell)$ extending $\varphi$ and which is still an irreducible representation. We can then apply \cite[Theorem (1.9)]{Sup95}, and the classification of Steinberg endomorphisms \cite[22.1--22.2]{TesMal11} gives the result.
\end{proof}

It remains to recall that
\begin{itemize}
\item $\Spin_n(\F_\lambda)$, $G_2(\F_\lambda)$ and $\SU_n(\F')$ fix a nondegenerate bilinear or unitary form on $V$ (see the proof of Theorem \ref{thm:KlcIntegralMono} page \pageref{page:preserveSymBilForm}; note that we use here that $T\to\SL_n(\F_\lambda)$ is the natural representation).
\item $\SU_n(\F')$ is not contained in $\Sp_n(\F_\lambda)$ if $n\ge 4$ (with respect to the standard inclusion $\SU_n(\F')\le\SL_n(\F_\lambda)$): take an orthogonal matrix corresponding to an even permutation exchanging the indices of two distinct entries of the matrix of the symplectic form.
\end{itemize}

This concludes the proof of Proposition \ref{prop:LieCharell}. \qed

\begin{table}
  \centering
  \begin{tabular}{c|c}
    $S$&$|\Out(S)|$\\ \hline
    $A_l(r)$&
    $\begin{cases}
      2a&:l=1\\
      2a&:l\ge 2\text{ even}\\
      2a(l+1,r-1)&:l\ge 3\text{ odd}
    \end{cases}$
\\
$\prescript{2}{}{A}_l(r)$&$\begin{cases}
  2a&:l\text{ even}\\
  2a(l+1,r+1)&:l\text{ odd}
\end{cases}
$\\
$B_l(r),C_l(r)$&$2a$\\
$D_l(r)$&$
\begin{cases}
  12a&:l=4\\
  8a&:l>4\text{ even}\\
  8a&:l\ge 5\text{ odd, }r\equiv 1\pmod{4}\\
  4a&:l\ge 5\text{ odd, }r\equiv 3\pmod{4}
\end{cases}$\\
$\prescript{2}{}{D}_l(r)$&$\begin{cases}
  4a&:r\equiv 1\pmod{4}\\
  8a&:r\equiv 3\pmod{4}, \ l\text{ odd}\\
  4a&:r\equiv 3\pmod{4}, \ l\text{ even}
\end{cases}
$\\
$\prescript{3}{}{D}_4(r)$&$a$\\
$E_6(r)$&$\begin{cases}
  6a&:r\equiv 1\pmod{3}\\
  2a&:r\equiv 2\pmod{3}
\end{cases}$\\
$\prescript{2}{}{E}_6(r)$&$\begin{cases}
  2a&:r\equiv 1\pmod{3}\\
  6a&:r\equiv 2\pmod{3}
\end{cases}$\\
$E_7(r)$&$2a$\\
$E_8(r)$&$a$\\
$F_4(r)$&$a$\\
$G_2(r)$&$a$
  \end{tabular}
  \caption{Outer automorphism groups of finite simple groups of Lie type over $\F_r$, with $r=q^a$ odd, $q>3$.}
  \label{table:OutS}
\end{table}
\subsection{Further classification theorems}\label{subsec:furtherClassThms}

Let $K$ be an \emph{algebraically closed} field of characteristic $\ell\ge 0$ and let $G$ be a classical group over $K$ (e.g. $G=\SL_n(K)$ or $G=\Sp_n(K)$) with associated vector space $V$.\\

Saxl and Seitz \cite{SaxlSeitz97} classified maximal closed subgroups $H\le G$ of positive dimension acting irreducibly on $V$ and containing a regular unipotent element of $G$, using the generalization of Aschbacher's result to algebraically closed fields by Liebeck-Seitz \cite{LiebSei98}. In particular, this generalizes \cite[Theorem (1.9)]{Sup95}. Note that Proposition \ref{prop:geometricSubgroups} is analogous to \cite[Proposition 2.1]{SaxlSeitz97}.

We remark that our classification theorem \ref{thm:classification} over $\F_\lambda$ cannot be simply deduced from \cite[Theorem B]{SaxlSeitz97} by descent. Indeed, taking a maximal (proper) subgroup $H\le\SL_n(\F_\lambda)$ containing a regular unipotent element and acting irreducibly on $\overline\F_\lambda^n$, we do not know whether there exists a positive-dimensional closed subgroup $H'\le\SL_n(\overline\F_\lambda)$ containing $H$. Showing this would actually be more or less equivalent to the proof of Theorem \ref{thm:classification}: if $H$ is allowed to be $0$-dimensional in \cite[Theorem B]{SaxlSeitz97}, one has to consider almost simple subgroups and not only simple ones, which is the additional difficulty we need to deal with in the proof of Theorem \ref{thm:classification}.\\

More generally, Testerman and Zalesski \cite[Theorem 1.2]{TesZal13} show that connected reductive linear algebraic groups containing a unipotent element with a single Jordan block are irreducible. Combined with \cite{SaxlSeitz97}, this gives a classification of semisimple subgroups $H$ of simple algebraic groups $G$ containing a regular unipotent element of $G$ (\cite[Theorem 1.4]{TesZal13}).

\bibliographystyle{alpha}
\small
\bibliography{references}
\normalsize

\end{document}